\documentclass[11pt, reqno]{amsart}

\usepackage[utf8]{inputenc}
\usepackage[T1]{fontenc}
\usepackage{lmodern}  
\usepackage{microtype} 
\usepackage{float}
\usepackage{placeins}

\usepackage[margin=1in]{geometry}

\usepackage{amsmath,amssymb,amsfonts,mathtools}
\usepackage{mathrsfs}         
\usepackage{bm}               
\numberwithin{equation}{section}  

\usepackage{amsthm}
\theoremstyle{plain}
  \newtheorem{theorem}{Theorem}[section]
  \newtheorem{lemma}[theorem]{Lemma}
  
  \newtheorem{corollary}[theorem]{Corollary}
\theoremstyle{definition}

\theoremstyle{remark}
  \newtheorem{remark}[theorem]{Remark}

\usepackage{graphicx}
\usepackage{subcaption}
\usepackage{tikz}
\usepackage{pgfplots}
  \pgfplotsset{compat=newest}
\usepackage[colorlinks=true,linkcolor=blue,citecolor=blue,urlcolor=blue]{hyperref}
\usepackage[nameinlink,noabbrev]{cleveref}

\usepackage{enumitem}
  \setlist[itemize]{topsep=4pt,partopsep=2pt,itemsep=2pt}
  \setlist[enumerate]{topsep=4pt,partopsep=2pt,itemsep=2pt}

\usepackage{listings}
\usepackage[ruled,vlined]{algorithm2e}

\usepackage{url}
\usepackage{xcolor}

\title{Pr\"ufer Transformation and Spectral Analysis for a Sturm--Liouville-Type Equation}
\author{Shalmali Bandyopadhyay, F. Ay\c{c}a \c{C}etinkaya, Tom Cuchta}
\date{}

\begin{document}

\begin{abstract}
We study a second-order differential equation involving a quasi-derivative, leading to a non-self-adjoint Sturm--Liouville-type problem with four coefficient functions. To analyze this equation, we develop a generalized Pr\"ufer transformation that expresses solutions in terms of amplitude and phase variables. We further prove the monotonicity of eigenfunction zeros with respect to the spectral parameter and derive upper and lower bounds for the eigenvalues.
\end{abstract}
\maketitle
\noindent \textbf{Keywords:} Pr\"ufer transformation, Sturm--Liouville theory, non-self-adjoint operators, quasi-derivative, spectral analysis, eigenvalue bounds, comparison theorems, phase function

\noindent \textbf{Mathematics Subject Classification (2020):} 34L15, 34B24, 34L05

\section{Introduction}

We consider the second-order differential equation
\begin{equation}
\label{1}
-\left[p(x)(y'(x) + s(x)y(x))\right]' + s(x)p(x)(y'(x) + s(x)y(x)) + q(x)y(x) = 0,
\end{equation}
which is a non-self-adjoint Sturm--Liouville-type {equation} problem. We develop a generalized Pr\"ufer transformation that expresses solutions in terms of amplitude and phase variables. 

Sturm--Liouville theory is a fundamental pillar of mathematical analysis with applications spanning from quantum mechanics and vibration analysis to signal processing and mathematical physics. The classical Sturm--Liouville {equation} problem 
\begin{equation}
-[p(x)y']' + q(x)y = \lambda w(x)y,\label{classicalSL}
\end{equation}
with various boundary conditions has been extensively studied since the 19th century \cite{sturm1836, sturm1837}. The theory for \eqref{classicalSL} provides a complete characterization of eigenvalues and eigenfunctions, including orthogonality properties, completeness of eigenfunction expansions, and oscillation theorems \cite{coddington1955, titchmarsh1962}.

The classical approach to analyzing Sturm--Liouville problems relies on the self-adjoint nature of the differential operator to guarantee real eigenvalues and orthogonal eigenfunctions in an appropriate weighted inner product space \cite{zettl2005}. However, many practical applications give rise to problems that do not fit the standard self-adjoint framework. These include problems with complex coefficients, eigenparameter-dependent boundary conditions, or more general quasi-derivative operators. The Pr\"ufer transformation (\cite{prufer1926}) has proven to be a powerful tool for studying oscillation of solutions to second-order differential equations. Recent research used Pr\"ufer transformation techniques to establish comparison theorems, analyze oscillation behavior, and develop numerical methods for eigenvalue problems \cite{atkinson1964, brown1999, kong2001}. 

Pr\"ufer theory grew to more general settings. For instance, eigenvalue problems where the leading coefficient function changes sign \cite{kong1999, kong2003} and numerics of two-parameter Sturm--Liouville problems \cite{binding1993}. The application of Pr\"ufer transformations to periodic and semi-periodic boundary conditions has been studied by several authors \cite{menendez2001, harris2003}, demonstrating the versatility of this approach across different types of boundary value problems. 

The study of non-self-adjoint Sturm--Liouville problems has gained significant attention in recent years due to their relevance in modeling dissipative systems and problems with complex potentials \cite{bondarenko2020}. Unlike their self-adjoint counterparts, such problems may exhibit complex eigenvalues, non-orthogonal eigenfunctions, and more intricate spectral properties. Non-self-adjoint operators arise naturally in processes that proceed without conservation of energy, including problems with friction, the theory of open resonators, and problems of inelastic scattering. The spectral theory of such operators presents unique challenges since traditional variational methods which rely on the minimization of energy functionals and real-valued eigenvalues with orthogonal eigenfunctions are not directly applicable to \eqref{1}. Consequently, researchers have developed specialized techniques including resolvent analysis, asymptotic expansions, and modified variational approaches \cite{davies2007}.

Despite the extensive literature on Sturm--Liouville theory, several important gaps remain in our understanding and computational capabilities. While significant progress has been made in studying non-self-adjoint problems or quasi-derivative operators independently, the literature lacks an analysis that combines both features. Equation \eqref{1} has a non-self-adjoint operator with a quasi-derivative structure involving four coefficient functions, which has not yet been systematically analyzed using Pr\"ufer transformation techniques. Existing Pr\"ufer transformations are designed for self-adjoint problems or specific classes of non-self-adjoint equations. 

{We introduce a novel Pr\"ufer transformation adapted to the modified derivative structure of \eqref{1},} resulting in a nonlinear system that captures the essential spectral properties of the original problem. We establish existence and uniqueness results for the generalized Pr\"ufer system and derive monotonicity properties for eigenfunction zeros with respect to spectral parameters. We obtain sharp upper and lower bounds for eigenvalues using variational techniques adapted to our non-self-adjoint setting. 

The rest of the paper is organized as follows. In Section \ref{sec:prufer}, we develop the generalized Pr\"ufer coordinates and establish the fundamental equivalence between the original differential equation \eqref{1} and its associated Pr\"ufer system, proving existence and uniqueness properties through Lipschitz continuity analysis. Section \ref{sec:spectral} focuses on the spectral properties and applications of the Pr\"ufer phase equation, beginning with the decoupling of amplitude and phase variables and demonstrating how the phase function completely determines solution zeros. We then analyze eigenvalue problems including the monotonicity property of eigenfunction zeros with respect to the spectral parameter. The section concludes with variational techniques that yield sharp upper and lower bounds for eigenvalues, and presents an eigenvalue criterion expressed entirely in terms of the phase function.

\section{Generalized Pr\"ufer Coordinates and Associated Differential System}
\label{sec:prufer}
The Pr\"ufer substitution provides a powerful technique for transforming second-order differential equations into a system of first-order equations, enabling deeper analysis of oscillatory behavior and spectral properties. In this section, we develop the generalized Pr\"ufer coordinates for \eqref{1} and establish the fundamental equivalence between {the differential equation} and its associated Pr\"ufer system.
\subsection{The Pr\"ufer transformation}
{For equation \eqref{1}, we introduce the Pr\"ufer substitution defined by}
\begin{equation}\label{2} 
p(x) \big(y^\prime (x)+s(x)y(x)\big)=-r(x)\cos\theta (x),
\end{equation} 
\begin{equation} \label{3} y(x)=r(x)\sin\theta (x),
\end{equation}
where $r$ and $\theta$ denote the \textit{amplitude} and \textit{phase} variables, respectively, given by
\begin{equation} \label{r}
r^2 (x)=y^2 (x)+\big[p(x)\big(y^\prime (x)+s(x)y(x)\big)\big]^2
\end{equation}
and 
\begin{equation} \label{theta}
\theta(x)=\arctan \bigg(\dfrac{y(x)}{-p(x) \big(y^\prime (x)+s(x)y(x)\big)}\bigg).
\end{equation}

Our goal is to derive a system of differential equations for $r$ and $\theta$ that is equivalent to \eqref{1}. For convenience, define 
\begin{equation} \label{*}
u(x)=y^\prime (x)+s(x)y(x).
\end{equation}
Then \eqref{2} becomes 
\begin{equation} \label{**}
p(x)u(x)=-r(x)\cos  \theta (x) .
\end{equation}
Now, consider \eqref{3} with 
\begin{equation} \label{4}
u(x)=-\dfrac{r(x)}{p(x)}\cos\theta(x).
\end{equation}
Differentiate \eqref{3} to obtain
\[y^\prime (x)=r^\prime (x) \sin\theta(x) +r(x)\theta'(x) \cos\theta(x).\]
From \eqref{3}, \eqref{*}, and \eqref{4}, we have 
\[y'(x) = u(x)-s(x)y(x) = -\dfrac{r(x)}{p(x)}\cos \theta (x)-s(x)r(x)\sin \theta(x).\]
Therefore, 
\begin{equation} \label{5}
r^\prime (x) \sin\theta(x)+r(x)\theta'(x) \cos\theta(x)=-\dfrac{r(x)}{p(x)}\cos\theta(x)-s(x)r(x)\sin\theta(x).
\end{equation}
Now substitute \eqref{2} and \eqref{3} into \eqref{1} and simplify to get
\begin{equation} \label{6}
r^\prime (x) \cos\theta(x)-r(x)\theta'(x)\sin\theta(x)-s(x)r(x)\cos\theta(x)+q(x)r(x)\sin\theta(x)=0.
\end{equation}
Multiply \eqref{5} by $\sin\theta(x)$, multiply \eqref{6} by $\cos\theta(x)$, and add together the resulting equations to attain
\begin{align} \label{7}
r'(x) ={}&
r(x)\Bigg[
-\frac{\cos\theta(x)\sin\theta(x)}{p(x)}
- s(x)\sin^2\theta(x) \nonumber\\
&\qquad
+ s(x)\cos^2\theta(x)
- q(x)\sin\theta(x)\cos\theta(x)
\Bigg].
\end{align}
To isolate $\theta^\prime(x)$, multiply \eqref{5} by $\cos\theta(x)$, \eqref{6} by $\sin\theta(x)$, and subtract the resulting equations to get 
\begin{align} \label{8}
\theta'(x)=-\dfrac{\cos^2\theta(x)}{p(x)} - 2s(x)\sin\theta(x)\cos\theta(x)+q(x)\sin^2\theta(x).
\end{align}

\subsection{The fundamental equivalence}
The system \eqref{7}--\eqref{8} is called the \textit{Pr\"ufer system} associated with \eqref{1}. It provides a complete characterization of the solutions to the original equation, as formalized in the following fundamental theorem.

\begin{theorem}[Equivalence of Pr\"ufer system and original system]
\label{thm:equivalence}
If $r$ and $\theta$ are solutions to \eqref{7}--\eqref{8}, then the function $y$ defined by \eqref{3} solves \eqref{1}. Also, conversely, if $y$ is a nontrivial $C^2$ solution to \eqref{1}, then there exist functions $r$ and $\theta$ satisfying \eqref{7}--\eqref{8} such that \eqref{2} and \eqref{3} hold with $r(x)\neq 0$ for $x \in [a,b]$. 
\end{theorem}

\begin{proof}
First suppose $r(x)$ and $\theta (x)$ solve \eqref{7}--\eqref{8} and define $y(x)$ by \eqref{3}. We verify that $y(x)$ satisfies \eqref{1} by direct substitution. Computing the derivative of \eqref{2} and multiplying by $-1$ yields
\[-\dfrac{\mathrm{d}}{\mathrm{d}x}\Big[ p(x)\big(y'(x)+s(x)y(x)\big)\Big]=\big(r(x) \cos\theta(x)\big)^\prime=r'(x) \cos\theta(x)-r(x)\theta'(x)\sin\theta(x).\]
Substituting \eqref{2} and \eqref{3} into the remaining terms of \eqref{1} reveals
\[s(x)p(x)(y'(x)+s(x)y(x)) = -r(x)s(x)\cos\theta(x),\]
and
\[q(x)y(x)=q(x)r(x)\sin\theta(x).\]

The left-hand side of \eqref{1} becomes
\begin{equation} 
r'(x)\cos\theta(x)-r(x)\theta'(x)\sin\theta(x)-r(x)s(x)\cos\theta(x)+q(x)r(x)\sin\theta(x)=0.
\label{thm2.1formula}
\end{equation}
Now substitute \eqref{7} and \eqref{8} into \eqref{thm2.1formula} and simplify to confirm that $y$ satisfies \eqref{1}.

Conversely, suppose $y(x)$ is a nontrivial $C^2$ solution to \eqref{1} on interval $[a,b]$. We define functions $r(x)$ and $\theta (x)$ on $[a,b]$ using \eqref{r}--\eqref{theta}. By construction, the Pr\"ufer substitution \eqref{2}--\eqref{3} is satisfied. We now verify that $r(x)$ and $\theta (x)$ satisfy \eqref{7}--\eqref{8}.

From \eqref{2}, we have $y(x) = r(x) \sin\theta(x)$. Differentiating with respect to $x$ yields
\begin{equation}\label{eq:y-prime}
y'(x) = r'(x)\sin\theta(x) + r(x)\theta'(x)\cos\theta(x).
\end{equation}
From \eqref{3}, we have $p(x)(y'(x) + s(x)y(x)) = -r(x)\cos\theta(x)$, which gives
\begin{equation}\label{eq:py}
p(x)y'(x) = -r(x)\cos\theta(x) - p(x)s(x)y(x).
\end{equation}
Substituting $y(x) = r(x)\sin\theta (x)$ into \eqref{eq:py}, we obtain
\begin{equation}\label{eq:py-expanded}
p(x)y'(x) = -r(x)\cos\theta(x) - p(x)s(x)r(x)\sin\theta(x),
\end{equation}
and multiplying equation \eqref{eq:y-prime} by $p(x)$ and equating with \eqref{eq:py-expanded} gives
\begin{equation}\label{eq:intermediate-1}
p(x)r'(x)\sin\theta(x) + p(x)r(x)\theta'(x)\cos\theta(x) + p(x)s(x)r(x)\sin\theta(x) + r(x)\cos\theta(x) = 0.
\end{equation}
Dividing by $p$ yields
\begin{equation}\label{eq:star-star}
r'(x)\sin\theta(x) + r(x)\theta'(x)\cos\theta(x) + s(x)r(x)\sin\theta(x) = -\frac{r(x)\cos\theta(x)}{p(x)}.
\end{equation}

We derive a second equation by differentiating \eqref{3} with respect to $x$. From \eqref{1} and the Pr\"ufer substitution, after algebraic manipulation we obtain
\begin{equation}\label{eq:system-ii}
r'(x)\cos\theta (x)- r(x)\theta'(x)\sin\theta(x) = r(x)s(x)\cos\theta(x) - q(x)r(x)\sin\theta(x). \tag{ii}
\end{equation}
To verify \eqref{eq:system-ii}, we use \eqref{1} and the Pr\"ufer substitution. Using the relations
\[
p(x)y'(x) = -r(x)\cos\theta(x) - p(x)s(x)r(x)\sin\theta(x) \quad \text{and} \quad y (x)= r(x)\sin\theta(x)
\]
we have 
\[
y'(x) = r'(x)\sin\theta(x) + r(x)\theta'(x)\cos\theta(x).
\]
The equation \eqref{1} can be written as
\begin{equation*}
-[p(x)(y'(x) + s(x)y(x))]' + s(x)p(x)(y'(x) + s(x)y(x)) + q(x)y(x) = 0.
\end{equation*}
Substituting the Pr\"ufer expressions and simplifying yields \eqref{eq:system-ii}.

Now, to obtain equation \eqref{8}, we eliminate $r'(x)$ from equations \eqref{eq:star-star} and \eqref{eq:system-ii}. Multiply \eqref{eq:star-star} by $\cos\theta(x)$ and \eqref{eq:system-ii} by $\sin\theta(x)$:
\begin{align*}
r'(x)\sin\theta(x)\cos\theta(x) + r(x)\theta'(x)\cos^2\theta(x) &= -\frac{r(x)\cos^2\theta(x)}{p(x)} - s(x)r(x)\sin\theta(x)\cos\theta(x)\\
r'(x)\cos\theta(x)\sin\theta (x)- r(x)\theta'(x)\sin^2\theta(x) &= r(x)s(x)\cos\theta(x)\sin\theta(x) - q(x)r(x)\sin^2\theta(x).
\end{align*}
Subtracting these equations:
\begin{equation*}
\begin{aligned}
r(x)\theta'(x)\bigl(\cos^2\theta(x) + \sin^2\theta(x)\bigr)
={}&
-\frac{r(x)\cos^2\theta(x)}{p(x)}
- s(x)r(x)\sin\theta(x)\cos\theta(x) \\
&\quad
- r(x)s(x)\cos\theta(x)\sin\theta(x)
+ q(x)r(x)\sin^2\theta(x).
\end{aligned}
\end{equation*}
Simplifying:
\begin{equation*}
r(x)\theta'(x) = -\frac{r(x)\cos^2\theta(x)}{p(x)} - 2s(x)r(x)\sin\theta(x)\cos\theta(x)+ q(x)r(x)\sin^2\theta(x).
\end{equation*}
Dividing by $r(x)$ (which is nonzero for a nontrivial solution):
\begin{equation}\label{eq:theta-prime-derived}
\theta'(x) = -\frac{\cos^2\theta(x)}{p(x)} -2s(x)\sin\theta(x)\cos\theta(x)+ q(x)\sin^2\theta(x),
\end{equation}
which is equation \eqref{8}.

To obtain equation \eqref{7}, we eliminate $\theta'$ from equations \eqref{eq:star-star} and \eqref{eq:system-ii}. Multiply \eqref{eq:star-star} by $\sin\theta(x)$ and \eqref{eq:system-ii} by $\cos\theta(x)$, then add:
\begin{align*}
r'(x)\sin^2\theta(x) + r(x)\theta'(x)\cos\theta(x)\sin\theta(x) &= -\frac{r(x)\cos\theta(x)\sin\theta(x)}{p(x)} - s(x)r(x)\sin^2\theta(x)\\
r'(x)\cos^2\theta(x) - r(x)\theta'(x)\sin\theta(x)\cos\theta(x) &= r(x)s(x)\cos^2\theta(x) - q(x)r(x)\sin\theta(x)\cos\theta(x).
\end{align*}
Adding these equations:
\begin{equation*}
\begin{aligned}
r'(x)\bigl(\sin^2\theta(x) + \cos^2\theta(x)\bigr)
={}&
-\frac{r(x)\cos\theta(x)\sin\theta(x)}{p(x)}
- s(x)r(x)\sin^2\theta(x) \\
&\quad
+ r(x)s(x)\cos^2\theta(x)
- q(x)r(x)\sin\theta(x)\cos\theta(x).
\end{aligned}
\end{equation*}
Simplifying:
\begin{equation*}
r' (x)= r(x)\left[-\frac{\cos\theta(x)\sin\theta(x)}{p(x)} + s(x)(\cos^2\theta(x) - \sin^2\theta(x)) - q(x)\sin\theta(x)\cos\theta(x)\right],
\end{equation*}
which is equation \eqref{7}, as was to be shown.
\end{proof}

\subsection{Existence and uniqueness}
To ensure the well-posedness of the Pr\"ufer system, we establish the Lipschitz continuity of the phase equation, which guarantees unique solutions.
\begin{lemma}[Lipschitz Continuity]
\label{lem22}
Let 
\[
f(x,\theta)=-\dfrac{\cos^2\theta}{p(x)}-2s(x)\sin\theta\cos\theta+q(x) \sin^2\theta,\]
where the coefficient functions $p, s, q \in C[a,b]$ and $p>0$ on $[a,b]$. Then for each fixed $x \in [a,b]$, the function $\theta \mapsto f(x,\theta)$ is Lipschitz continuous with Lipschitz constant:
\[\hat{L}=\sup_{x \in [a,b]} \left\{ \frac{1}{p(x)} + 2|s(x)| + |q(x)| \right\}.\]
\end{lemma}
\begin{proof}
Compute the partial derivative of $f(x,\theta(x))$ with respect to $\theta(x)$ to get
\[\frac{\partial f(x,\theta)}{\partial \theta} = \frac{\sin (2\theta)}{p(x)} - 2 s(x) \cos(2\theta) + q(x) \sin(2\theta).\]
Using the triangle inequality, we obtain the stated Lipschitz constant, completing the proof.
\end{proof}
\noindent This Lipschitz continuity immediately implies the following existence and uniqueness result.
\begin{corollary}[Global existence and uniqueness]
Under the continuity assumptions of Lemma~\ref{lem22}, for any initial value $\theta(x_0)=\theta_0$ with $x_0 \in [a,b]$, the phase equation \eqref{8} has a unique solution $\theta$ defined on all of $[a,b]$.
\end{corollary}

\section{Spectral properties and applications of the Pr\"ufer phase equation}
\label{sec:spectral}
In this section, we demonstrate how the phase function $\theta(x)$ completely determines the zeros of solutions.

The phase function $\theta(x)$ is determined by solving \eqref{8}, and hence the amplitude equation \eqref{7} becomes a first-order differential equation of the form
\begin{equation*} 
r^\prime(x) = r(x) F(x),
\end{equation*}
where
\[
\begin{aligned}
F(x) ={}&
-\frac{\cos\theta(x)\,\sin\theta(x)}{p(x)}
- s(x)\sin^2\theta(x) \\
&\quad
+ s(x)\cos^2\theta(x)
- q(x)\sin\theta(x)\cos\theta(x).
\end{aligned}
\]
Hence $r(x)$ is an exponential which never vanishes. Therefore by \eqref{3}, the zeros of solutions are completely determined by the phase function, independent of the amplitude. Since $y(x) = r(x)\sin\theta(x)$ and $r > 0$, we have $y(x) = 0$ precisely when $\sin\theta(x) = 0$, which occurs when $\theta(x) \in \pi \mathbb{Z}$.

This enables direct comparison of oscillatory behavior between different equations by comparing their respective phase equations. The Pr\"ufer approach gives a framework for establishing oscillation and comparison theorems.
\begin{remark}
The Sturm oscillation theorem does not directly generalize to \eqref{1}. 
To illustrate this, we introduce the notation
\[
y^{[1]} := p(x)\bigl(y'(x)+s(x)y(x)\bigr),
\]
and, for $i=1,2$, set $y_i^{[1]} := p(x)\bigl(y_i'(x)+s(x)y_i(x)\bigr)$.
With this notation, equation~\eqref{1} is equivalent to
\begin{equation}\label{sturm_osc}
-\bigl(y_i^{[1]}\bigr)' + s(x)y_i^{[1]} + q_i(x)y_i = 0 .
\end{equation}

Assume that $q_1<q_2$, and let $x_1$ and $x_2$ be two consecutive zeros of a nontrivial solution $y_1$ of \eqref{sturm_osc}.
Define the Wronskian-like function
\[
W(x) := y_1(x)y_2^{[1]}(x) - y_2(x)y_1^{[1]}(x).
\]
A direct computation gives
\[
\begin{aligned}
W'(x)
&= y_1'(x)y_2^{[1]}(x) + y_1(x)\bigl(y_2^{[1]}(x)\bigr)'
   - y_2'(x)y_1^{[1]}(x) - y_2(x)\bigl(y_1^{[1]}(x)\bigr)' .
\end{aligned}
\]

Using \eqref{sturm_osc}, we have
\[
\bigl(y_i^{[1]}\bigr)' = s(x)y_i^{[1]} + q_i(x)y_i ,
\]
and substituting this into the expression for $W'$ yields
\[
W'(x) = \bigl(q_2(x)-q_1(x)\bigr)y_1(x)y_2(x).
\]
Integrating from $x_1$ to $x_2$ gives
\[
W(x_2)-W(x_1)
= \int_{x_1}^{x_2} \bigl(q_2(x)-q_1(x)\bigr)y_1(x)y_2(x)\,\mathrm{d}x .
\]

By the definition of $W$, this identity becomes
\[
- y_2(x_2)y_1^{[1]}(x_2) + y_2(x_1)y_1^{[1]}(x_1)
= \int_{x_1}^{x_2} \bigl(q_2(x)-q_1(x)\bigr)y_1(x)y_2(x)\,\mathrm{d}x .
\]

Suppose that $y_2$ does not vanish on $(x_1,x_2)$.
Without loss of generality, assume that $y_1$ and $y_2$ are positive on $(x_1,x_2)$.
Since $y_1(x_1)=0$ and $y_1>0$ on $(x_1,x_2)$, it follows that $y_1$ is increasing at $x_1$, and hence
\[
y_1^{[1]}(x_1) = p(x_1)y_1'(x_1).
\]
Similarly, since $x_2$ is also a zero of $y_1$, we obtain
\[
y_1^{[1]}(x_2) = p(x_2)y_1'(x_2).
\]
Because $p>0$, the functions $y_1$ and $y_1^{[1]}$ have the same sign at both $x_1$ and $x_2$, with
\[
y_1^{[1]}(x_1)>0,
\qquad
y_1^{[1]}(x_2)<0.
\]
Consequently, $W(x_1)<0$ and $W(x_2)>0$.
By continuity, there exists $x^*\in(x_1,x_2)$ such that $W(x^*)=0$.

By definition of $W$, this implies
\[
y_1(x^*)y_2^{[1]}(x^*) = y_2(x^*)y_1^{[1]}(x^*).
\]
If one could conclude that $y_2(x^*)=0$, this would yield a contradiction and the proof would be complete. However, if $y_1^{[1]}(x^*)=0$, then
\[
y_1(x^*)y_2^{[1]}(x^*) = 0.
\]
Since, by assumption, $y_1(x^*)>0$, we must have $y_2^{[1]}(x^*)=0$, that is,
\[
p(x^*)\bigl(y_2'(x^*)+s(x^*)y_2(x^*)\bigr)=0,
\]
or equivalently,
\[
s(x^*) = -\frac{y_2'(x^*)}{y_2(x^*)}.
\]
This condition cannot be imposed a priori, since the choice of $s$ determines the solution $y_2$.
\end{remark}

\subsection{Eigenvalue Problems and Spectral Analysis}
Replace $q$ with $q - \lambda \omega$ in \eqref{1} for a given weight function $\omega > 0$ to get the eigenvalue problem
\begin{equation} \label{14}
-\left(p(x)\left(y'(x) + s(x)y(x)\right)\right)' + s(x)p(x)\left(y'(x) + s(x)y(x)\right) + q(x)y(x) = \lambda \omega(x)y(x).
\end{equation}

The theorem below shows that zeros of solutions to the eigenvalue problem move leftward as the eigenvalue parameter increases, a fundamental monotonicity property in Sturm-Liouville theory. The proof uses the Pr\"ufer transformation to convert the second-order eigenvalue problem into a first-order phase equation, where the dependence on the eigenvalue parameter $\lambda$ becomes explicit and monotonic. By applying comparison theorems to the resulting phase equations, we show that larger eigenvalues produce faster phase accumulation, causing zeros to occur at earlier positions in the interval.

\begin{theorem}[Monotonicity of zeros with respect to \(\lambda\)]
\label{thm:monotonicity}
Let \( p, s, q, \omega \in C([a,b]) \) with \( p > 0 \) and \( \omega > 0 \) on \( [a,b] \). Consider the eigenvalue problem \eqref{14} with the boundary condition \( y(a) = 0 \) and an initial slope \( y'(a) > 0 \). For each \( \lambda \), let \( y_\lambda(x) \) be the corresponding nontrivial solution, and denote its successive zeros by
\[
x_1(\lambda) < x_2(\lambda) < \cdots,
\]
with each \( x_k(\lambda) \in (a,b] \). Then, for any \( \mu > \lambda \), the zeros of \( y_\mu(x) \) satisfy
\[
x_k(\mu) < x_k(\lambda) \quad \text{for all } k.
\]
\end{theorem}

\begin{proof}
For each value of the spectral parameter $\lambda$, the eigenvalue problem \eqref{14} with the initial condition $y(a) = 0$, $y'(a) > 0$ determines a unique solution $y_\lambda(x)$. Let $y_\lambda(x) = r_\lambda(x)\sin(\theta_\lambda(x))$ denote its Pr\"ufer representation. The initial condition $y(a) = 0$ implies $\theta_\lambda(a) \in \pi\mathbb{Z}$; we normalize by taking $\theta_\lambda(a) = \theta_\mu(a) = \alpha$ for some fixed value $\alpha$.

Substituting $q(x) - \lambda\omega(x)$ for $q(x)$ in the phase equation \eqref{8}, we obtain
\[
\theta'(x) = f_\lambda(x, \theta) := -\frac{\cos^2\theta(x)}{p(x)} - 2s(x)\sin\theta(x)\cos\theta(x) + \bigl(q(x) - \lambda\omega(x)\bigr)\sin^2\theta(x).
\]
Here we view $f_\lambda(x, \theta)$ as a function defined on $[a,b] \times \mathbb{R}$, with $\theta$ treated as an independent variable (not yet a solution).
We now compare $f_\mu$ and $f_\lambda$. For any point $(x, \theta) \in [a,b] \times \mathbb{R}$,
\begin{align*}
f_\mu(x, \theta) - f_\lambda(x, \theta) 
&= \bigl(q(x) - \mu\omega(x)\bigr)\sin^2\theta - \bigl(q(x) - \lambda\omega(x)\bigr)\sin^2\theta \\
&= (\lambda - \mu)\omega(x)\sin^2\theta.
\end{align*}
Since $\mu > \lambda$, $\omega(x) > 0$, and $\sin^2\theta \geq 0$, we conclude that
\[
f_\mu(x, \theta) \leq f_\lambda(x, \theta) \quad \text{for all } (x, \theta) \in [a,b] \times \mathbb{R},
\]
with strict inequality whenever $\sin\theta \neq 0$.
Now, by the standard comparison theorem for scalar first-order ODEs \cite{coddington1955}, if two solutions $\theta_\mu$ and $\theta_\lambda$ satisfy $\theta_\mu' = f_\mu(x, \theta_\mu)$ and $\theta_\lambda' = f_\lambda(x, \theta_\lambda)$ with $\theta_\mu(a) = \theta_\lambda(a)$, and if $f_\mu(x, \theta) \leq f_\lambda(x, \theta)$ for all $(x, \theta)$, then
\[
\theta_\mu(x) \leq \theta_\lambda(x) \quad \text{for all } x \geq a.
\]
Moreover, since the inequality $f_\mu < f_\lambda$ is strict whenever $\sin\theta \neq 0$, we obtain the strict inequality $\theta_\mu(x) < \theta_\lambda(x)$ for all $x > a$. Note that since $\theta'|_{\theta=0} = -1/p(x) < 0$, the phase functions are decreasing; hence $\theta_\mu < \theta_\lambda$ implies that $\theta_\mu$ reaches each successive multiple of $\pi$ at a smaller value of $x$.

The zeros of $y(x) = r(x)\sin(\theta(x))$ occur precisely when $\theta(x) \in \pi\mathbb{Z}$. Since $\theta_\mu(x) < \theta_\lambda(x)$ for $x > a$ and both phase functions are decreasing, $\theta_\mu$ reaches each successive negative multiple of $\pi$ strictly before $\theta_\lambda$ does. Defining
\[
x_k(\lambda) := \min\{x > a : \theta_\lambda(x) = \alpha - k\pi\}, \quad x_k(\mu) := \min\{x > a : \theta_\mu(x) = \alpha - k\pi\},
\]
we conclude that $x_k(\mu) < x_k(\lambda)$ for every $k \in \mathbb{N}$.
\end{proof}

\subsection{Eigenvalue Bounds}

The Pr\"ufer framework enables precise bounds on eigenvalues through variational techniques. To analyze the eigenvalues of the boundary value problem~\eqref{14} with Dirichlet boundary conditions $y(a) = y(b) = 0$, we first derive the associated Rayleigh quotient. Multiplying~\eqref{14} by $y$ and integrating by parts yields
\[
\mathcal{R}[y] = \frac{\displaystyle\int_a^b p(x)\bigl(y'(x) + s(x)y(x)\bigr)^2\,\mathrm{d}x + \int_a^b q(x)y^2(x)\,\mathrm{d}x}{\displaystyle\int_a^b \omega(x)y^2(x)\,\mathrm{d}x}.
\]
Introducing the substitution $v(x) := e^{\bar{s}(x)}y(x)$, where $\bar{s}(x) = \int_a^x s(t)\,\mathrm{d}t$, we have $y = e^{-\bar{s}}v$ and $y' + sy = e^{-\bar{s}}v'$. Under this change of variables, the Rayleigh quotient transforms to
\begin{equation}\label{RayleightQuotient}
\mathcal{R}[v] = \frac{\displaystyle\int_a^b \tilde{p}(x)(v'(x))^2\,\mathrm{d}x + \int_a^b Q(x)v^2(x)\,\mathrm{d}x}{\displaystyle\int_a^b \tilde{\omega}(x)v^2(x)\,\mathrm{d}x},
\end{equation}
where the transformed coefficient functions are
\begin{align}
\tilde{p}(x) &:= p(x)e^{-2\bar{s}(x)}, \label{eq:ptilde}\\
Q(x) &:= q(x)e^{-2\bar{s}(x)}, \label{associatedcoefficient}\\
\tilde{\omega}(x) &:= \omega(x)e^{-2\bar{s}(x)}. \label{eq:omegatilde}
\end{align}
Note that in the ratio $\frac{Q(x)}{\tilde{\omega}(x)}$, the exponential factors cancel:
\[
\frac{Q(x)}{\tilde{\omega}(x)} = \frac{q(x)e^{-2\bar{s}(x)}}{\omega(x)e^{-2\bar{s}(x)}} = \frac{q(x)}{\omega(x)}.
\]
Thus, the constants $m$ and $M$ appearing in the bounds below depend only on the original coefficients $q$ and $\omega$. Importantly, the change of variables $v = e^{\bar{s}}y$ transforms the original non-self-adjoint problem~\eqref{14} into the self-adjoint Dirichlet problem with Rayleigh quotient~\eqref{RayleightQuotient}, allowing us to apply variational methods to derive lower and upper bounds for the eigenvalues $\lambda_n$.

\begin{remark}{\rm The proof of the following theorems relies on the minimax principle, also known as the Courant--Fischer theorem~\cite{fischer1905,courant1920}. For a self-adjoint boundary value problem with Rayleigh quotient $\mathcal{R}[v]$, the $n^{\text{th}}$ eigenvalue $\lambda_n$ satisfies
\[
\lambda_n = \min_{\substack{S \subset H_0^1(a,b) \\ \dim(S) = n}} \max_{\substack{v \in S \\ v \neq 0}} \mathcal{R}[v] = \max_{\substack{T \subset H_0^1(a,b) \\ \dim(T) = n-1}} \min_{\substack{v \in T^\perp \\ v \neq 0}} \mathcal{R}[v],
\]
where $T^\perp$ denotes the orthogonal complement of the subspace $T$ in $H_0^1(a,b)$.}
\end{remark}

\begin{theorem}[Lower bound on the $n^{\text{th}}$ eigenvalue]
\label{th:lower-bound}
Let $\tilde{p}(x) = p(x)e^{-2\bar{s}(x)}$ and $\tilde{\omega}(x) = \omega(x)e^{-2\bar{s}(x)}$, where $\bar{s}(x) = \int_a^x s(t)\,\mathrm{d}t$. Suppose $f \in C([a,b])$ satisfies $0 < f(x) \leq \tilde{p}(x)$ for all $x \in [a,b]$, and there exists a constant $c > 0$ such that $\tilde{\omega}(x)f(x) < c$ on $[a,b]$. Define
\[
m := \min_{x \in [a,b]} \frac{q(x)}{\omega(x)}.
\]
Then the $n^{\text{th}}$ eigenvalue of~\eqref{14} satisfies
\[
\lambda_n \geq \frac{n^2\pi^2}{c\left(\displaystyle\int_a^b \frac{1}{f(x)}\,\mathrm{d}x\right)^2} + m, \qquad n = 1, 2, 3, \ldots
\]
\end{theorem}

\begin{proof}
Write the Rayleigh quotient~\eqref{RayleightQuotient} as
\[
\mathcal{R}[v] = \frac{\displaystyle\int_a^b \tilde{p}(v')^2\,\mathrm{d}x}{\displaystyle\int_a^b \tilde{\omega}\,v^2\,\mathrm{d}x} + \frac{\displaystyle\int_a^b Q\,v^2\,\mathrm{d}x}{\displaystyle\int_a^b \tilde{\omega}\,v^2\,\mathrm{d}x}.
\]

Since 
\[
m = \min_{x \in [a,b]} \frac{Q(x)}{\tilde{\omega}(x)} = \min_{x \in [a,b]} \frac{q(x)}{\omega(x)},
\]
we have $Q(x) \geq m\,\tilde{\omega}(x)$ for all $x \in [a,b]$, which gives
which gives
\[
\frac{\displaystyle\int_a^b Q\,v^2\,\mathrm{d}x}{\displaystyle\int_a^b \tilde{\omega}\,v^2\,\mathrm{d}x} \geq m.
\]

Moreover, since $f \leq \tilde{p}$, we have
\[
\frac{\displaystyle\int_a^b \tilde{p}(v')^2\,\mathrm{d}x}{\displaystyle\int_a^b \tilde{\omega}\,v^2\,\mathrm{d}x} \geq \frac{\displaystyle\int_a^b f(v')^2\,\mathrm{d}x}{\displaystyle\int_a^b \tilde{\omega}\,v^2\,\mathrm{d}x}.
\]
Now, consider the self-adjoint comparison problem
\begin{equation}\label{comparison}
-(f(x)\varphi'(x))' = \tilde{\lambda}\,\frac{\varphi(x)}{f(x)}, \qquad \varphi(a) = \varphi(b) = 0,
\end{equation}
whose Rayleigh quotient is
\[
\tilde{\mathcal{R}}[\varphi] = \frac{\displaystyle\int_a^b f(\varphi')^2\,\mathrm{d}x}{\displaystyle\int_a^b \frac{\varphi^2}{f}\,\mathrm{d}x}.
\]
To determine the eigenvalues of the comparison problem~\eqref{comparison}, we apply the Liouville substitution
\[
u = \int_a^x \frac{\mathrm{d}t}{f(t)}.
\]
Then $\frac{\mathrm{d}u}{\mathrm{d}x} = \frac{1}{f(x)}$, and by the chain rule we obtain
\[
\varphi'(x)
= \frac{\mathrm{d}\varphi}{\mathrm{d}u}\,\frac{\mathrm{d}u}{\mathrm{d}x}
= \frac{1}{f(x)}\frac{\mathrm{d}\varphi}{\mathrm{d}u}.
\]
Consequently,
\[
f(x)\varphi'(x) = \frac{\mathrm{d}\varphi}{\mathrm{d}u},
\qquad
\bigl(f(x)\varphi'(x)\bigr)'
= \frac{\mathrm{d}}{\mathrm{d}x}\left(\frac{\mathrm{d}\varphi}{\mathrm{d}u}\right)
= \frac{1}{f(x)}\frac{\mathrm{d}^2\varphi}{\mathrm{d}u^2}.
\]
Substituting these expressions into the differential equation~\eqref{comparison} yields
\[
-\frac{1}{f(x)}\frac{\mathrm{d}^2\varphi}{\mathrm{d}u^2}
= \frac{\tilde{\lambda}}{f(x)}\varphi,
\]
which simplifies to the constant-coefficient equation
\[
\frac{\mathrm{d}^2\varphi}{\mathrm{d}u^2} + \tilde{\lambda}\varphi = 0.
\]
Under this change of variables, the boundary conditions $\varphi(a)=\varphi(b)=0$ transform to
\[
\varphi(0)=0,
\qquad
\varphi(L)=0,
\]
where
\[
L=\int_a^b \frac{\mathrm{d}t}{f(t)}.
\]
This is the classical Dirichlet eigenvalue problem for the second derivative on the interval $(0,L)$, whose eigenvalues are
\[
\tilde{\lambda}_n = \frac{n^2\pi^2}{L^2}
= \frac{n^2\pi^2}{\left(\displaystyle\int_a^b \frac{1}{f(x)}\,\mathrm{d}x\right)^2},
\qquad n=1,2,3,\ldots
\]
Combining the preceding estimates with the assumption $\tilde{\omega}(x)f(x)<c$, we have shown that for every
$v\in H_0^1(a,b)\setminus\{0\}$,
\[
\mathcal{R}[v]
\ge \frac{1}{c}\,\tilde{\mathcal{R}}[v] + m,
\]
where
\[
\tilde{\mathcal{R}}[v]
= \frac{\displaystyle\int_a^b f(x)(v'(x))^2\,\mathrm{d}x}
{\displaystyle\int_a^b \frac{v(x)^2}{f(x)}\,\mathrm{d}x}
\]
is the Rayleigh quotient associated with the comparison problem~\eqref{comparison}.
By the standard eigenvalue comparison principle, if $\mathcal{R}_1[v]\ge \mathcal{R}_2[v]$ for all admissible $v$,
then the $n$-th eigenvalue of the first problem is at least as large as the the $n$-th eigenvalue of the second. Applying this principle with $\mathcal{R}_1=\mathcal{R}$ and
$\mathcal{R}_2[v]=\frac{1}{c}\tilde{\mathcal{R}}[v]+m$, and using the fact that the $n$-th eigenvalue of the
comparison problem~\eqref{comparison} is $\tilde{\lambda}_n$, we conclude that
\[
\lambda_n \ge \frac{\tilde{\lambda}_n}{c}+m
= \frac{n^2\pi^2}{c\left(\displaystyle\int_a^b \frac{1}{f(x)}\,\mathrm{d}x\right)^2}+m.
\] \qedhere
\end{proof}

\begin{theorem}[Upper bound on the $n^{\text{th}}$ eigenvalue]
\label{th:upper-bound}
Let $\tilde{p}(x) = p(x)e^{-2\bar{s}(x)}$ and $\tilde{\omega}(x) = \omega(x)e^{-2\bar{s}(x)}$, where $\bar{s}(x) = \int_a^x s(t)\,\mathrm{d}t$. Suppose $h \in C([a,b])$ satisfies $h(x) \geq \tilde{p}(x)$ for all $x \in [a,b]$, and define
\[
D := \min_{x \in [a,b]} \tilde{\omega}(x)h(x) > 0.
\]
Let
\[
M := \max_{x \in [a,b]} \frac{q(x)}{\omega(x)}.
\]
Then the $n^{\text{th}}$ eigenvalue of~\eqref{14} satisfies
\[
\lambda_n \leq \frac{n^2\pi^2}{D\left(\displaystyle\int_a^b \frac{1}{h(x)}\,\mathrm{d}x\right)^2} + M, \qquad n = 1, 2, 3, \ldots
\]
\end{theorem}

\begin{proof}
We again write the Rayleigh quotient~\eqref{RayleightQuotient} in the form
\[
\mathcal{R}[v]
= \frac{\displaystyle\int_a^b \tilde{p}\,(v')^2\,\mathrm{d}x}
       {\displaystyle\int_a^b \tilde{\omega}\,v^2\,\mathrm{d}x}
+ \frac{\displaystyle\int_a^b Q\,v^2\,\mathrm{d}x}
       {\displaystyle\int_a^b \tilde{\omega}\,v^2\,\mathrm{d}x}.
\]
Since
\[
M=\max_{x\in[a,b]}\frac{Q(x)}{\tilde{\omega}(x)}
   =\max_{x\in[a,b]}\frac{q(x)}{\omega(x)},
\]
we have $Q(x)\le M\,\tilde{\omega}(x)$ for all $x\in[a,b]$. Consequently,
\[
\frac{\displaystyle\int_a^b Q\,v^2\,\mathrm{d}x}
     {\displaystyle\int_a^b \tilde{\omega}\,v^2\,\mathrm{d}x}
\le M.
\]
Moreover, since $h\ge\tilde{p}$ on $[a,b]$, we have
\[
\frac{\displaystyle\int_a^b \tilde{p}\,(v')^2\,\mathrm{d}x}
     {\displaystyle\int_a^b \tilde{\omega}\,v^2\,\mathrm{d}x}
\le
\frac{\displaystyle\int_a^b h\,(v')^2\,\mathrm{d}x}
     {\displaystyle\int_a^b \tilde{\omega}\,v^2\,\mathrm{d}x}.
\]

We now consider the self-adjoint comparison problem
\begin{equation}\label{comparison_upper}
-(h(x)\psi'(x))' = \mu\,\frac{\psi(x)}{h(x)},
\qquad
\psi(a)=\psi(b)=0.
\end{equation}
Multiplying by $\psi$ and integrating by parts yields the Rayleigh quotient
\[
\tilde{\mathcal{R}}[\psi]
=
\frac{\displaystyle\int_a^b h\,(\psi')^2\,\mathrm{d}x}
     {\displaystyle\int_a^b \frac{\psi^2}{h}\,\mathrm{d}x}.
\]
Applying the Liouville substitution
\[
u=\int_a^x \frac{\mathrm{d}t}{h(t)}
\]
reduces this problem to the Dirichlet eigenvalue problem for the second derivative on the interval
\(
(0,L)
\),
where
\(
L=\int_a^b \frac{\mathrm{d}t}{h(t)}.
\)
The corresponding eigenvalues are
\[
\mu_n
=
\frac{n^2\pi^2}{L^2}
=
\frac{n^2\pi^2}
{\left(\displaystyle\int_a^b \frac{1}{h(x)}\,\mathrm{d}x\right)^2},
\qquad n=1,2,3,\ldots
\]
The hypothesis $\tilde{\omega}h\ge D$ implies
\[
\tilde{\omega}\ge \frac{D}{h} \quad \text{for all } x\in[a,b].
\]
Therefore,
\[
\int_a^b \tilde{\omega}\,v^2\,\mathrm{d}x
\ge
D\int_a^b \frac{v^2}{h}\,\mathrm{d}x,
\]
and it follows that
\[
\frac{\displaystyle\int_a^b h\,(v')^2\,\mathrm{d}x}
     {\displaystyle\int_a^b \tilde{\omega}\,v^2\,\mathrm{d}x}
\le
\frac{1}{D}
\frac{\displaystyle\int_a^b h\,(v')^2\,\mathrm{d}x}
     {\displaystyle\int_a^b \frac{v^2}{h}\,\mathrm{d}x}=\dfrac{\tilde{\mathcal{R}}[v]}{D}
\]
where $\tilde{\mathcal{R}}[v]$ is the Rayleigh quotient for the comparison problem~\eqref{comparison_upper}. 

Let $\psi_1,\dots,\psi_n$ denote the first $n$ eigenfunctions of ~\eqref{comparison_upper} and define
\[
V=\mathrm{span}\{\psi_1,\dots,\psi_n\}.
\]
Then $\dim V=n$, and for every $v\in V\setminus\{0\}$ we have
\(
\tilde{\mathcal{R}}[v]\le\mu_n
\).
Combining the previous inequalities, it follows that
\[
\mathcal{R}[v]
\le
\frac{1}{D}\,\tilde{\mathcal{R}}[v] + M
\le
\frac{\mu_n}{D} + M,
\qquad v\in V\setminus\{0\}.
\]
Taking the maximum over $v\in V\setminus\{0\}$ and invoking the minimax characterization of eigenvalues yields
\[
\lambda_n
=
\min_{\substack{S\subset H_0^1(a,b)\\ \dim S=n}}
\max_{\substack{v\in S\\ v\neq 0}}
\mathcal{R}[v]
\le
\max_{\substack{v\in V\\ v\neq 0}}
\mathcal{R}[v]
\le
\frac{\mu_n}{D} + M.
\]
Substituting the expression for $\mu_n$ gives
\[
\lambda_n
\le
\frac{n^2\pi^2}
     {D\left(\displaystyle\int_a^b \frac{1}{h(x)}\,\mathrm{d}x\right)^2}
+ M,
\]
which completes the proof.
\qedhere
\end{proof}

\subsection{The phase function criterion}
Perhaps the most elegant application of the Pr\"ufer method is the complete characterization of eigenvalues through the phase function. At this point, it is appropriate to explain why the function space \( H_0^1(a,b) \) is used in the application of the minimax principle. This Sobolev space consists of all functions \( v \in H^1(a,b) \) that vanish at the endpoints, i.e.\ \( v(a) = v(b) = 0 \). This is the natural setting for the variational formulation of the eigenvalue problem since it ensures the square-integrability of both \( v \) and \( v' \), making the Rayleigh quotient well-defined. Moreover, \( H_0^1(a,b) \) reflects the boundary conditions and carries a Hilbert space structure under which the self-adjoint differential operator admits a discrete spectrum. The minimax principle characterizes the \(n\)-th eigenvalue as the smallest maximum of the Rayleigh quotient over all \(n\)-dimensional subspaces of \( H_0^1(a,b) \), justifying its central role in the proof.

\begin{theorem}[Eigenvalue criterion via the phase function]
\label{thm:phasefunction}
Let $p, s, q, \omega \in C([a,b])$ with $p > 0$ and $\omega > 0$ on $[a,b]$. For each $\lambda \in \mathbb{R}$, let $\theta(x; \lambda)$ denote the solution to the initial value problem
\begin{equation}\label{phase_lambda}
\begin{aligned}
\theta'(x)
={}&
-\frac{\cos^2\theta(x)}{p(x)}
- 2s(x)\sin\theta(x)\cos\theta(x) \\
&\quad
+ \bigl(q(x)-\lambda\omega(x)\bigr)\sin^2\theta(x),
\end{aligned}
\end{equation}
\[
\theta(a;\lambda)=0.
\]
and let $r(x; \lambda)$ denote the solution to
\begin{equation}\label{amp_lambda}
\begin{aligned}
r'(x)
={}&
r(x)\Bigg[
-\frac{\cos\theta(x)\,\sin\theta(x)}{p(x)}
- s(x)\sin^2\theta(x)
+ s(x)\cos^2\theta(x)\\
&\qquad
- \bigl(q(x)-\lambda\omega(x)\bigr)\sin\theta(x)\cos\theta(x)
\Bigg],
\end{aligned}
\end{equation}
\[
r(a;\lambda)=r_0>0.
\]
Then $\lambda$ is an eigenvalue of the boundary value problem~\eqref{14} with Dirichlet conditions $y(a) = y(b) = 0$ if and only if $\theta(b; \lambda) = n\pi$ for some positive integer $n \in \mathbb{N}$. In this case, the corresponding eigenfunction is $y(x) = r(x; \lambda)\sin\theta(x; \lambda)$.
\end{theorem}

\begin{proof}
Consider the eigenvalue problem~\eqref{14} with Dirichlet boundary conditions
\[
-[p(x)(y' + sy)]' + sp(y' + sy) + qy = \lambda\omega y, \quad y(a) = 0, \quad y(b) = 0.
\]
By Theorem~\ref{thm:equivalence}, any nontrivial solution $y(x)$ of this problem admits a Pr\"ufer representation of the form \eqref{2}, \eqref{3}, where the amplitude $r(x)$ is strictly positive on $[a,b]$ and satisfies~\eqref{7}, while the phase $\theta(x)$ satisfies~\eqref{8}. Replacing $q$ by $q - \lambda\omega$ in these equations yields the phase equation~\eqref{phase_lambda} and the amplitude equation~\eqref{amp_lambda} corresponding to the eigenvalue parameter $\lambda$. 

Since  $r(x)>0$ for all $x \in [a,b]$ \eqref{2} shows that 
\[
y(x) = 0 \quad \text{if and only if} \quad \sin\theta(x) = 0
\]
which is equivalent to $\theta(x) \in \pi\mathbb{Z}$. In particular, the boundary condition $y(a) = 0$ implies $\theta(a) \in \pi\mathbb{Z}$. We fix the normalization by choosing
\[
\theta(a; \lambda) = 0.
\]
With this normalization, the boundary condition $y(b) = 0$ is equivalent to
\[
\theta(b; \lambda) = n\pi
\]
for some integer $n$. 

If $\theta(b; \lambda) = 0$, then $\theta(x; \lambda) \equiv 0$ would be a solution only if the right-hand side of~\eqref{phase_lambda} vanishes identically when $\theta = 0$. However, evaluating the right-hand side of the phase equation~\eqref{phase_lambda} at $\theta=0$ gives 
\[
\theta'(x)\big|_{\theta=0} = -\frac{1}{p(x)} \neq 0,
\]
which is nonzero for all $x \in [a,b]$. Hence $\theta$ remain identically zero, and therefore 
$\theta(b; \lambda) = 0$ is impossible when $b > a$. 

The integer $n$ must be positive.  If $\theta(a) = 0$ and $\theta(b) = n\pi$ with $n < 0$ then replacing $y$ by $-y$ , which is also an eigenfunction corresponding to the same eigenvalue, shifts the phase by $\pi$ and produces a positive integer. Thus we may restrict attention to $n \in \mathbb{N}$. 

Suppose now that $\theta(b; \lambda) = n\pi$ for some $n \in \mathbb{N}$. Let $\theta(x) = \theta(x; \lambda)$ be the solution of the phase equation~\eqref{phase_lambda} with initial condition $\theta(a) = 0$. Solve the amplitude equation~\eqref{amp_lambda} with an initial condition $r(a) = r_0 > 0$. Since the amplitude equation is linear in $r$, positivity of the initial value implies $r(x) > 0$ for all $x \in [a,b]$. Define
\[
y(x) = r(x)\sin\theta(x).
\]
By Theorem~\ref{thm:equivalence}, this function $y$ solves the differential equation~\eqref{14}. Moreover,
\[
y(a) = r(a)\sin(0) = 0, \qquad y(b) = r(b)\sin(n\pi) = 0,
\]
so the Dirichlet boundary conditions are satisfied. The function $y$  is nontrivial because $\theta$ varies continuously from $0$ to $n\pi$ with $n \geq 1$, and therefore there exist points $x \in (a,b)$ for which $\theta(x) \notin \pi\mathbb{Z}$. At such points $\sin\theta(x) \neq 0$, and since $r(x) > 0$, we have $y(x) \neq 0$.  Thus $\lambda$ is an eigenvalue.

Conversely, suppose that $\lambda$ is an eigenvalue with a corresponding nontrivial eigenfunction $y$. By Theorem~\ref{thm:equivalence}, $y$ admits a Pr\"ufer representation $y = r\sin\theta$ with $r(x) > 0$ on $[a,b]$. The boundary conditions $y(a) = 0$ and $y(b) = 0$ imply that $\theta(a), \theta(b) \in \pi\mathbb{Z}$. With the normalization $\theta(a) = 0$, it follows that $\theta(b) = n\pi$ for some integer $n$. Since $y$ is nontrivial, the case $n=0$ is excluded, and by convention we take $n > 0$. This completes the proof.
\end{proof}

This criterion transforms the eigenvalue problem into a root-finding problem: eigenvalues are precisely those values of $\lambda$ for which $\theta(b; \lambda)$ equals an integer multiple of $\pi$. This provides both theoretical insight and computational advantages for numerical eigenvalue algorithms.

The elegance of this approach lies in its reduction of a two-point boundary value problem to the solution of an initial value problem, making it particularly suitable for numerical implementation and theoretical analysis.

\section{Conclusion}
The theoretical contribution of this work opens several promising avenues for future research, including numerical implementation of this generalized Pr\"ufer transformation, numerical analysis of convergence properties for eigenvalue computation in non-self-adjoint settings, and computational validation of the eigenvalue bounds derived in Theorems \ref{th:lower-bound} and \ref{th:upper-bound} across various parameter regimes.

A natural and significant extension of this work involves adapting the generalized Pr\"ufer transformation to dynamic equations on time scales. Such an extension is particularly compelling because time scales provide a unified framework that encompasses both continuous and discrete settings, potentially revealing new connections between differential and difference equations in spectral theory. The time scale approach may also yield computational advantages in certain applications where the underlying physical system naturally evolves on a non-uniform time domain.

\section*{Acknowledgements}
The author (S. Bandyopadhyay) was supported by the AMS-Simons Travel Grants program. This support enabled collaborative research visits that were instrumental in this paper.

{\footnotesize  
\medskip
\medskip
\vspace*{1mm} 
 
\noindent {\it Shalmali Bandyopadhyay}\\  
The University of Tennessee at Martin\\
Martin, TN\\
E-mail: {\tt sbandyo5@utm.edu}\\ \\  

\noindent {\it F. Ay\c{c}a \c{C}etinkaya}\\  
The University of Tennessee at Chattanooga\\
Chattanooga, TN\\
E-mail: {\tt fatmaayca-cetinkaya@utc.edu}\\ \\

\noindent {\it Tom Cuchta}\\  
Marshall University\\
Huntingdon, WV\\
E-mail: {\tt cuchta@marshall.edu}\\ \\

\end{document}